\documentclass[a4paper,10pt]{amsart}
\usepackage{amsmath, amssymb, amsthm,amsfonts}
\usepackage{verbatim}
\usepackage{hyperref}
\usepackage{xypic}
\usepackage{graphicx}
\usepackage{xcolor}
\usepackage{mathtools}
\newcommand{\T}[1]{\textup{#1}}

\newcommand{\cstar}{\ensuremath{\mathrm C^\ast}}
\newcommand{\Star}{{}^\ast}

\usepackage{pdflscape}

\newcounter{theorem}
\newtheorem{thm}[theorem]{Theorem}
\newtheorem{lemma}[theorem]{Lemma}
\newtheorem{prop}[theorem]{Proposition}
\newtheorem{cor}[theorem]{Corollary}
\newtheorem{defn}[theorem]{Definition}
\newtheorem{notation}[theorem]{Notation}

\newtheorem*{remark*}{Remark}
\newtheorem{remark}[theorem]{Remark}
\newtheorem{example}[theorem]{Example}

\numberwithin{equation}{section}
\numberwithin{theorem}{section}

\newcommand{\Q}{\mathbb{Q}}
\newcommand{\Z}{\mathbb{Z}}
\newcommand{\Tor}{\mathrm{Tor}_1^{\Z}}
\newcommand{\LTor}{\mathrm{LTor}_1^{\mathbb Z}}
\newcommand{\RTor}{\mathrm{RTor}_1^{\mathbb Z}}

\newcommand{\id}{\mathrm{id}}

\newcommand{\AT}[1]{}
\newcommand{\AS}[1]{}
\newcommand{\DE}[1]{}

\begin{document}

\title[Corrigendum: approximately inner flip]{Corrigendum to ``K-theoretic characterization of $\cstar$-algebras with approximately inner flip''}
\author[D.\ Enders]{Dominic Enders}
\author[A.\ Schemaitat]{Andr\'e Schemaitat}
\author[A.\ Tikuisis]{Aaron Tikuisis}
\thanks{AT supported by an NSERC discovery grant.}

\begin{abstract}
An error in the original paper is identified and corrected. The \cstar-algebras with approximately inner flip, which satisfy the UCT, are identified (and turn out to be fewer than what is claimed in the original paper). The action of the flip map on K-theory turns out to be more subtle, involving a minus sign in certain components. To this end, we introduce new geometric resolutions for \cstar-algebras, which do not involve index shifts in K-theory and thus allow for a more explicit description of the quotient map in the K\"unneth formula for tensor products.
\end{abstract}

\maketitle

\section{Introduction: the problem and corrections}

The errors in \cite{aifK} stem from a mistake in the proof of \cite[Lemma 4.1]{aifK}, where it is said ``The map $\alpha_{A,B}$ is explicitly described in \cite[Section 23.1]{Blackadar:KBook}, and it is apparent from this description that $\alpha_{B,A} \circ \sigma_{K_\ast(A),K_\ast(B)} = K_\ast(\sigma_{A,B}) \circ \alpha_{A,B}$, i.e., the first square in (4.2) commutes.''

In fact, the following example shows that 
\[ \alpha_{B,A} \circ \sigma_{K_*(A),K_*(B)} \neq K_*(\sigma_{A,B}) \circ \alpha_{A,B}, \]
when restricted to the component $K_1(A) \otimes K_1(B)$ (and thus with codomain $K_0(B \otimes A)$).

\begin{example}
\label{ex:bott-element}
Let $A=B \coloneqq C(\mathbb T)$ and let $u \in C(\mathbb T)$ be the canonical generator, so that $[u]_1$ is a generator of $K_1(C(\mathbb T)) \cong \mathbb Z$.
We identify $C(\mathbb T)\otimes C(\mathbb T)$ with $C(\mathbb T^2)$, and under this identification, the flip map $\sigma_{C(\mathbb T),C(\mathbb T)}$ corresponds to swapping coordinates:
\[ (\sigma_{C(\mathbb T),C(\mathbb T)}(f))(w,z) = f(z,w). \]
We have $K_0(C(\mathbb T^2)) \cong \mathbb Z^2$, generated by $[1_{C(\mathbb T^2)}]_0$ and the Bott element $b$.
Note that $K_0(\sigma_{C(\mathbb T),C(\mathbb T)})(b)= -b$\footnote{Since the Bott element in $K_0(C_0((0,1)^2))$ arises from a clutching construction, any reflection on $(0,1)^2$, and in particular the coordinate swap, i.e., the flip map, sends this Bott element to its inverse. Using the canonical embedding $C_0((0,1)^2) \to C(\mathbb T^2)$, it follows that the same is true for the flip map and Bott element for $C(\mathbb T^2)$.}
and $\alpha_{C(\mathbb T),C(\mathbb T)}([u]_1 \otimes [u]_1) = b$. Thus,
\begin{align*}
 \alpha_{C(\mathbb T),C(\mathbb T)}(\sigma_{K_*(C(\mathbb T)),K_*(C(\mathbb T))}([u]_1\otimes [u]_1))
& = 
 \alpha_{C(\mathbb T),C(\mathbb T)}([u]_1 \otimes [u]_1) \\
& = b, 
\end{align*}
whereas
\begin{align*}
	 K_*(\sigma_{C(\mathbb T),C(\mathbb T)})(\alpha_{C(\mathbb T),C(\mathbb T)}([u]_1 \otimes [u]_1))
 & = K_*(\sigma_{C(\mathbb T),C(\mathbb T)})(b) \\
& = -b \neq b.
\end{align*}
\end{example}

We shall show that the square on the left in the K\"unneth flip formula (\cite[Equation (4.2)]{aifK}) commutes \emph{exactly up to a minus sign} on $K_1(A)\otimes K_1(B)$ (and commutes without the minus sign on the other three components).
This has a knock-on effect that the square in the right in \cite[Equation (4.2)]{aifK} only commutes up to minus signs in appropriate components (full details in Lemma \ref{lem:KunnethFlipCorrected} below), and further implications in saying precisely which classifiable $\cstar$-algebras have approximately inner half-flip.
Here is an example demonstrating where the minus sign occurs for the torsion part of the K\"unneth flip formula.

\begin{example}
	Set $A = B \coloneqq \mathcal O_{n+1}$. By the K{\"u}nneth formula, we have that $\beta_{\mathcal O_n, \mathcal O_n} \colon K_1(\mathcal O_{n+1} \otimes \mathcal O_{n+1}) \to \Tor(\Z/n\Z, \Z/n\Z) \cong \Z/n\Z$ is an isomorphism and it is easy to check (for example by verifying the proof of \cite[Proposition 5.1]{aifK}) that $\eta_{\Z/n\Z,\Z/n\Z}$ is the identity. On the other hand by \cite[Proposition 3.5]{FHRT} the class of 
	$$
		u \coloneqq \sum_{i=1}^{n+1} s_i \otimes s_i^*
	$$
	generates $K_1(\mathcal O_{n+1} \otimes \mathcal O_{n+1})$, where $\{s_i\}_{i=1}^{n+1}$ is the canonical set of generators for $\mathcal O_{n+1}$. Since $\sigma_{\mathcal O_{n+1},\mathcal O_{n+1}}(u) = u^* = u^{-1}$ it follows that $K_1(\sigma_{\mathcal O_{n+1}, \mathcal O_{n+1}}) = -\id$ so that $(\ref{eq:KunnethFlipDiag2})$ indeed commutes in the  odd degree.
\end{example}

We introduce new techniques in the proof of the corrected version of commutation of the right square in \cite[Eq.\ (4.2)]{aifK}.
By making use of Kirchberg algebras, we are able to obtain free resolutions that do not involve index shifts; this makes the $\mathrm{Tor}$-related computations more conceptual and easier.

Another error was made in the listing of the $\cstar$-algebras with approximately inner flip, where it is incorrectly stated that the supernatural number $n$ in \cite[Theorem 2.2]{aifK} (and related results) must be of infinite type.
We take the opportunity to correct this error as well.

We list here all results which are incorrect in \cite{aifK} and their corrections.
In the next sections, we shall prove the corrected statements.
(Note that $\Q_1=\Z$.)

\begin{thm}[{Correction to \cite[Theorem 2.2]{aifK}}]
\label{thm:MainThmCorrected}
Let $A$ be a separable, unital $\cstar$-algebra with strict comparison, in the $\T{UCT}$ class, which is either infinite or quasidiagonal.
The following are equivalent.
\begin{enumerate}
\item $A$ has approximately inner flip; \stepcounter{enumi}
\item $A$ has asymptotically inner flip;
\item $A$ is simple, nuclear, has at most one trace and $K_0(A)\oplus K_1(A)$ (as a graded, unordered group) is isomorphic to one of $0\oplus \Q_m/\Z$ or $\Q_n \oplus \Q_m/\Z$, where $m$ and $n$ are supernatural numbers with $m$ of infinite type and such that $m$ divides $n$;
\item $A$ is Morita equivalent to one of:
\begin{enumerate}
\item $\mathbb C$;
\item $\mathcal E_{n,1,m}$;
\item $\mathcal E_{n,1,m} \otimes \mathcal O_\infty$;
\item $\mathcal F_{1,m}$,
\end{enumerate}
where in $\mathrm{(b)}$-$\mathrm{(d)}$, $m$ and $n$ are supernatural numbers with $m$ of infinite type and such that $m$ divides $n$.
\end{enumerate}
\end{thm}

In \cite[Theorem 2.2]{aifK}, it was claimed that the above conditions are also equivalent to (ii) $A \otimes A$ has approximately inner flip.
However, it now transpires that this is not the case, as for example $\mathcal O^\infty$ does not have approximately inner flip, but $\mathcal O^\infty \otimes \mathcal O^\infty$ (which is Morita equivalent to $\mathcal O_\infty$) does have approximately inner flip.

The following is the correct version of the fundamental lemma from \cite{aifK} which relates the K\"unneth formula to the flip map.
To give a statement that takes the index-related behaviour into account, we use $\beta_{A,B}\colon K_{i+j+1}(A\otimes B) \to \Tor(K_i(A),K_j(B))$ to denote the K\"unneth formula map
\[ K_{i+j+1}(A\otimes B) \to \Tor(K_i(A),K_j(B))\oplus \Tor(K_{1-i}(A),K_{1-j}(B)), \]
composed with the coordinate projection onto $\Tor(K_i(A),K_j(B))$.

\begin{lemma}[{Correction to \cite[Lemma 4.1]{aifK}}]
\label{lem:KunnethFlipCorrected}
Let $A,B$ be separable $\cstar$-algebras in the $\textup{UCT}$ class.
For $i,j \in \Z /2 \Z$, the following diagrams commute:
\begin{equation}
\label{eq:KunnethFlipDiag1}
\begin{gathered}
\xymatrix{
K_i(A) \otimes K_j(B) \ar[r]^-{\alpha_{A,B}}\ar[d]_-{(-1)^{ij}\sigma_{K_i(A),K_j(B)}} & K_{i+j}(A\otimes B) \ar[d]^-{K_{i+j}(\sigma_{A,B})} \\
K_j(B) \otimes K_i(A) \ar[r]^-{\alpha_{B,A}} & K_{i+j}(B\otimes A) 
}
\end{gathered}
\end{equation}
and
\begin{equation}
\label{eq:KunnethFlipDiag2}
\begin{gathered}
\xymatrix{
K_{i+j+1}(A\otimes B) \ar[r]^-{\beta_{A,B}}\ar[d]_-{K_{i+j+1}(\sigma_{A,B})} & \Tor(K_i(A), K_j(B)) \ar[d]^-{(-1)^{1+ij}\eta_{K_i(A),K_j(B)}}\\
K_{i+j+1}(B\otimes A) \ar[r]^-{\beta_{B,A}} & \Tor(K_j(B), K_i(A)).
} 
\end{gathered}
\end{equation}
The map $\beta_{A,B}$ is as in Definition \ref{def:beta} and fits into the K{\"u}nneth formula.
\end{lemma}

\begin{thm}[{Correction to \cite[Theorem 5.2]{aifK}}]
\label{thm:KKsuff}
Let $A$ be a separable $\cstar$-algebra in the $\textup{UCT}$ class.
Suppose that $K_0(A)\oplus K_1(A)$ is one of the following graded groups:
\begin{enumerate}
\item $0 \oplus \Q_m/\Z$,where $m$ is a supernatural number of infinite type; or
\item $\Q_n \oplus \Q_m/\Z$, where $m$ and $n$ are supernatural numbers with $m$ of infinite type and such that $m$ divides $n$.
\end{enumerate}
Then the flip map $\sigma_{A,A}:A \otimes A \to A \otimes  A$ has the same KK-class as the identity map.
\end{thm}

\begin{cor}[{Correction to \cite[Corollary 5.3]{aifK}}]
\label{cor:Suff}
Let $m$ and $n$ be supernatural numbers such that $m$ has infinite type and $m$ divides $n$.
Then $\mathcal E_{n,1,m}$ and $\mathcal F_{1,m}$ have asymptotically inner flip.
\end{cor}

The next result is a modification of \cite[Theorem 6.1]{aifK}, giving correct restrictions on the K-theory of a classifiable $\cstar$-algebra with approximately inner flip (matching those in Theorem \ref{thm:MainThmCorrected}).

\begin{thm}
\label{thm:NecessaryCorrected}
Let $A$ be a separable $\cstar$-algebra in the $\T{UCT}$ class  such that $A$ has approximately inner flip.
Then $K_0(A)\oplus K_1(A)$ is isomorphic to one of the following graded groups:
\begin{enumerate}
\item $0 \oplus \Q_m/\Z$, where $m$ is a supernatural number of infinite type; or
\item $\Q_n \oplus \Q_m/\Z$, where $m$ and $n$ are supernatural numbers with $m$ of infinite type and such that $m$ divides $n$.
\end{enumerate}
\end{thm}

Although the proofs of \cite[Lemmas 6.2-6.4]{aifK} are problematic as they use \cite[Lemma 4.1]{aifK}, their statements are still correct, as we will prove.

The result \cite[Corollary 7.4]{aifK} is incorrect, as for example the $\cstar$-algebra $A \coloneqq \bigotimes_{p \text{ prime}} M_p$ has approximately inner flip but $A \otimes A$ is not self-absorbing.

\subsection*{Pull-backs}
In addition to the notation and preliminaries found in \cite{aifK}, we will make use of pull-backs in the category of $\cstar$-algebras; we remind the reader of their explicit realization here.
Given $\cstar$-algebras $A,B,C$ and $\Star$-homomorphisms $\phi \colon A \to C, \psi \colon B \to C$, the associated \emph{pull-back} is the $\cstar$-algebra
\[ D \coloneqq \{(a,b) \in A \oplus B : \phi(a)=\psi(b)\}, \]
and the associated pull-back diagram is the commuting diagram
\[ \xymatrix{
D \ar[r]\ar[d] & A \ar[d]^{\phi} \\ B\ar[r]_{\psi} & C, } \]
where the maps $D \to A$ and $D \to B$ are the restrictions of the coordinate projections on $A\oplus B$.\\

Throughout the paper, $\otimes$ denotes the minimal tensor product.

\section{Geometric resolutions for $\cstar$-algebras revisited}
\label{sec:geores}

In order to prove the K{\"u}nneth formula for tensor products, one needs a way to realize projective resolutions of the K-groups of a $\cstar$-algebra by maps on the level of $\cstar$-algebras. These so-called {\it geometric resolutions} were introduced by Rosenberg and Schochet in \cite{RS}. Their construction can be summarized as follows:\\

\noindent {\itshape Given any separable $\cstar$-algebra $A$, there exists an extension 
$$ 
0\to S(S^2A\otimes{\mathbb{K}}) \to E\to B\to 0
$$
of separable $\cstar$-algebras such that the associated K-theory six-term exact sequence degenerates into
\[
	\xymatrix{
	0 \ar[r] & K_*(E) \ar[r] & K_*(B) \ar[r]^-\partial & K_*(S(S^2A\otimes\mathbb{K})) \ar[r]& 0
	}
\]
and constitutes a projective resolution of $K_*(S(S^2A\otimes \mathbb{K}))\cong K_{*+1}(A)$.} 

\ \\
There is, however, one minor drawback to their approach. The projective resolutions obtained this way always involve a degree shift and a boundary map $\partial$. This can cause problems in situations where the K-groups carry additional structure, which is preserved by K-theory maps induced by $\Star$-homomorphisms, but not necessarily by the map $\partial$. This is, for example, the case in the study of K-theory for tensor products, where one needs to keep track of decompositions like $K_{i}(C\otimes D)=\bigoplus_{j\in\Z/2\Z} K_j(C)\otimes K_{i+j}(D)$ (for suitable $\cstar$-algebras $C,D$).

In this section we introduce new geometric resolutions, which do not involve these index shifts. This makes the K-theory computations for $\cstar$-tensor products more conceptual and allows for a more explicit description of the quotient map in the K\"unneth formula, see Proposition \ref{prop:kunnethformula}.
\par Another subtlety is that we ask our geometric resolutions to be semisplit, which is automatic for those  used by Rosenberg and Schochet as they use mapping cone sequences. The feature of being semisplit is crucial for preserving exactness when tensoring with a fixed \cstar-algebra with respect to the minimal tensor product, cf. Remark \ref{rem:semisplit}, but is further used to get a two-out-of-three property for the UCT, which is automatic for nuclear $\cstar$-algebras. Let us first prove this well-known result.

\begin{lemma}
	\label{lem:semisplit}
	Assume
	\[
		\xymatrix{
			0 \ar[r] & I \ar[r]^{\iota} & E \ar[r]^{\pi} & A \ar[r] & 0
		}
	\]
	is a semisplit extension of separable $\cstar$-algebras, i.e. $\pi$ admits a completely positive contractive (c.p.c.) split. Then, if two of the three algebras in the short exact sequence satisfy the $\T{UCT}$, so does the third.
\end{lemma}

\begin{proof}
	 Let $D$ be a $\sigma$-unital $\cstar$-algebra with $K_*(D) = 0$. Using \cite[Theorem 19.5.7]{Blackadar:KBook} and the fact that our extension is semisplit, there exists a six-term exact sequence
	 \[
	 	\xymatrix{
	 		KK(A,D) \ar[r] & KK(E,D) \ar[r] & KK(I,D) \ar[d] \\
	 		KK^1(I,D) \ar[u] &  KK^1(E,D) \ar[l] & KK^1(A,D) \ar[l]
	 	}
	 \]
	 If for example $A$ and $E$ satisfy the UCT, the above diagram simplifies to 
	  \[
	 	\xymatrix{
	 		0 \ar[r] & 0 \ar[r] & KK(I,D) \ar[d] \\
	 		KK^1(I,D) \ar[u] &  0 \ar[l] & 0 \ar[l]
	 	}
	 \]
	 It follows that $KK^*(I,D) = 0$ and, since $D$ was arbitrary, \cite[23.10.5 (iv)$\Rightarrow$(i)]{Blackadar:KBook} implies that $I$ satisfies the UCT, too. The same argument applies if we assume any other two $\cstar$-algebras in the extension to satisfy the UCT. 
\end{proof}

\begin{defn}
\label{def:geores}
	Let $A$ be a  $\cstar$-algebra. A semisplit short exact sequence 
	\[
		0 \to I \to E \to A \to 0
	\]
	is called a {\bf $\partial$-free geometric resolution} of $A$, if the associated K-theory six-term exact sequence degenerates into
	\[
		0 \to K_*(I) \to K_*(E) \to K_*(A) \to 0
	\]
	and constitutes a projective resolution for $K_*(A)$.
\end{defn}

\begin{lemma}
\label{lem:trick}
Let $\varphi\colon A\to B$ be a $\Star$-homomorphism. Then, there exists a $\cstar$-algebra $\hat{A}$ containing $A$ such that
\begin{enumerate}
	\item the inclusion $A\subseteq \hat{A}$ is a homotopy equivalence, 
	\item the map $\varphi$ extends to a surjective $\Star$-homomorphism $\hat{\varphi}$ which, in addition, admits a completely positive splitting $s$:
		\[
		\begin{xy}
		\xymatrix{
			A \ar[r]^\varphi \ar[d]_\subseteq & B \ar@/^1pc/[dl]^s \\
			\hat{A} \ar@{->>}[ur]^{\hat{\varphi}}
		}
		\end{xy}
		\]
\end{enumerate}
\end{lemma}

\begin{proof}
We make use of free products. Let $CB$ denote the cone over $B$, i.e., $CB = C_0(0,1] \otimes B$. Set $\hat{A}\coloneqq A\ast CB$ and $\hat{\varphi}\coloneqq \varphi\ast \T{ev}_1$. Then $\id_A\ast 0$ is a homotopy-inverse for the inclusion of $A$ into $\hat{A}$, proving the homotopy equivalence statement.
The extension $\hat{\varphi}$ is surjective by construction and admits a c.p.c. split by composing $B \to CB:   b \mapsto \iota \otimes  b$, where $\iota(t) := t$, with the inclusion $CB \to \hat{A}$.
\end{proof}

The next theorem shows that there exist sufficiently many $\partial$-free geometric resolutions.

\begin{thm}
\label{thm:existence-resolutions}
Let $A$ be a separable $\cstar$-algebra satisfying the $\T{UCT}$ and let  
\[
	0\to P_* \to Q_* \stackrel{p}\longrightarrow K_*(A) \to 0
\]
 be a countable, projective resolution of $K_*(A)$.
Then, there exists a $\partial$-free geometric resolution consisting of separable $\cstar$-algebras satisfying the $\T?{UCT}$, which induces the given resolution of $K_*(A)$. 
\end{thm}

\begin{proof}
Let $B$ be a unital UCT-Kirchberg algebra with $K_*(B)=Q_*\oplus(\mathbb{Z},0) = (Q_0 \oplus \Z, Q_1)$.
Let us identify $p \oplus \id$ with an element of $\mathrm{Hom}(K_*(B),K_*(\tilde A))$:
\[
	\xymatrix{
		 Q_*\oplus(\mathbb{Z},0)  \ar[r]^-{p\oplus\id}  & K_*(A)\oplus(\mathbb{Z},0) \\
		 K_*(B) \ar[u]_{\cong} \ar[r] & K_*(\tilde A) \ar[u]^\cong 
	}
\] 
Since $B$ satisfies the UCT,  we may lift $p \oplus \id$ to an element in $KK(B,\tilde{A}) \cong KK(B, \tilde A \otimes \mathcal O_\infty)$.
By  \cite[Theorem 8.2.1~(i)]{Rordam:Book}, this KK-element can be realized by a $\Star$-homomorphism
	$$
		\varphi\colon B\to (\tilde{A}\otimes\mathcal{O}_\infty)\otimes \mathbb{K}.
	$$
Now apply Lemma \ref{lem:trick} to obtain a semisplit extension
\[
	\begin{xy}
		\xymatrix{0 \ar[r] & I \ar[r] & \overline{B} \ar[r]^(.3){\overline{\varphi}} & \tilde{A}\otimes\mathcal{O}_\infty \otimes \mathbb{K} \ar[r] & 0}
	\end{xy}
\]
 such that $B \subseteq \overline B$ is a homotopy equivalence. After identifying $K_*(\overline B)$ with $K_*(B)$ and $K_*(\tilde A \otimes \mathcal O_\infty \otimes \mathbb K)$ with $K_*(\tilde A)$, we see that  $K_*(\overline{\varphi})$ is isomorphic to $p\oplus\id$. Since $\overline{B}\sim_{KK} B$ and $ \tilde{A}\otimes\mathcal{O}_\infty \otimes \mathbb{K} \sim_{KK} \tilde{A}$, we see that $\overline B$ and $\tilde{A}\otimes\mathcal{O}_\infty \otimes \mathbb{K}$ satisfy the UCT.  As the extension is semisplit, we find $I$ to satisfy the UCT by Lemma \ref{lem:semisplit}.
Next, we need to restrict this extension to $A\subseteq \tilde{A} \subseteq (\tilde{A}\otimes\mathcal{O}_\infty)\otimes \mathbb{K}$. We do this in two steps; first consider the subextension
\[
	\begin{xy}
		\xymatrix{0 \ar[r] & I \ar[r] & \overline{\varphi}^{-1}(\tilde{A}) \ar[r]^{\overline{\varphi}} & \tilde{A} \ar[r] & 0},
	\end{xy},
\]
which is still semiplit (with the restriction of the split from before). Hence by the same two-out-of-three argument we find $\overline{\varphi}^{-1}(\tilde{A})$ to satisfy the UCT. The five-lemma applied to the K-theory of this extension shows that  this subextension still induces $p\oplus \id$. More precisely, the following diagram commtues:
\[
	\xymatrix{
		0 \ar[r] & K_*(I) \ar@{=}[d] \ar[r] & K_*(\overline \varphi^{-1}(\tilde A)) \ar[r]^{K_*(\overline \varphi)} \ar[d]^{\cong}  & K_*(\tilde A) \ar[r]  \ar[d]^{\cong} &  0 \\
		0 \ar[r] & K_*(I) \ar[r] & K_*(\overline B) \ar[r]_-{K_*(\overline \varphi)} & K_*(\tilde A \otimes \mathcal O_\infty \otimes \mathbb K) \ar[r] & 0 
	}
\] 
Finally, we pass to the subextension
\[
	\begin{xy}
		\xymatrix{0 \ar[r] & I \ar[r] & \overline{\varphi}^{-1}(A) \ar[r]^-{\overline{\varphi}} & A \ar[r] & 0}
	\end{xy},
\]
which is again semiplit (again, by restricting the split from before) and therefore consists of $\cstar$-algebras satisfying the UCT.
Since $\overline{\varphi}^{-1}(\tilde{A}) / \overline{\varphi}^{-1}(A)$ is isomorphic to $\tilde{A}/A$, this extension models $p\colon Q_*\to K_*(A)$ on K-theory. 
This finishes the proof.
\end{proof}

The following will allow us to compare any two given $\partial$-free geometric resolutions. It further shows that our construction is natural with respect to $\Star$-homomorphisms, cf.~Remark \ref{rem:natural}.

\begin{lemma}
\label{lem:comparison}
Given two $\partial$-free geometric resolutions
\[
	\xymatrix@R-1em{
		0\ar[r]&I\ar[r]^-\iota &E\ar[r]^-\pi&A\ar[r]&0 ,
 \\
 		0\ar[r]&I'\ar[r]^-{\iota'} &E'\ar[r]^-{\pi'}&A'\ar[r]&0
	}
\]
of  $\cstar$-algebras $A,A'$ and a $\Star$-homomorphism $\varphi\colon A\to A'$, form the pullback 
\[
\xymatrix{
\overline E \ar[d]_-p \ar[r]^-{p'} & E' \ar[d]^-{\pi'} \\
E \ar[r]_-{\varphi\circ\pi} & A' 
}
\]
\AS{The notation $\overline E$ should not be confused with the notation from Lemma \ref{lem:trick}. I suggest to change this.}
Then, with $\overline \pi \coloneqq \pi\circ p$ and $\overline I \coloneqq \T{Ker}(\overline \pi)$, the following diagram commutes and the middle row is a $\partial$-free geometric resolution of $A$.
\begin{equation}
	\label{eq:pullback-resolution}
	\begin{aligned}
	\xymatrix{
		0 \ar[r] &  I' \ar[r] &  E' \ar[r]^{\pi'} & A' \ar[r] & 0 \\
		0 \ar[r] & \overline I \ar[r] \ar[u]_{q'} \ar[d]^q & \overline E \ar[r]^{\overline \pi} \ar[d]^p \ar[u]_{p'} & A \ar@{->}[r] \ar@{=}[d] \ar[u]_\varphi& 0  \\
		0 \ar[r] & I \ar[r] & E \ar[r]^{\pi} & A \ar[r] & 0 
	}
		\end{aligned}
\end{equation}
If the geometric resolutions of $A$ and $A'$ consist of  separable $\cstar$-algebras satisfying the $\T{UCT}$, then $\overline I$ and $\overline E$ are separable and satisfy the $\T{UCT}$.
\end{lemma}

\begin{proof}
The K-groups of $\overline{E}$ can be computed using the long exact Mayer-Vietoris sequence (cf.~\cite[Theorem 21.2.2]{Blackadar:KBook}):
\[
	\xymatrix{
	 K_{*-1}(A') \ar[r]^-\gamma & K_*(\overline{E}) \ar[rr]^-{p_*\oplus p'_*} & &  K_*(E)\oplus K_*(E') \ar[rr]^-{\begin{pmatrix}   (\varphi\circ\pi)_* \\  -\pi'_* \end{pmatrix} } & &  K_*(A') \ar[d]^-\gamma \\
	 \vdots \ar[u] &  & &  & & \vdots
	}
\]
Using surjectivity of $\pi'_*$ one finds $\gamma=0$ and $p_*$ to be surjective.\footnote{For $x \in K_*(E)$, there exists $y \in K_*(E')$ such that $\pi'_*(y)=(\varphi\circ\pi)_*(x)$. Thus, $(x,y) \in \T{Ker} \begin{pmatrix}   (\varphi\circ\pi)_* \\  -\pi'_* \end{pmatrix} =\mathrm{Im}(p_*\oplus p'_*)$, and so $x \in \mathrm{Im}(p_*)$.} In particular, $(\pi\circ p)_* = (\overline \pi)_*$ is also surjective and the K-theory six-term sequence associated the the middle row of \eqref{eq:pullback-resolution} therefore degenerates to
\[
	0\to K_*(\overline{I}) \to K_*(\overline{E}) \to K_*(A) \to 0.
\]
This is a projective resolution of $K_*(A)$ since, by exactness of the long sequence above, $K_*(\overline{E})$ is a subgroup of the free group $K_*(E)\oplus K_*(E')$ and is thereby free itself. Furthermore, if $s$ resp. $s'$ are c.p.c. splits for $\pi$ resp. $\pi'$, then the map 
$$
	\overline s \colon A \to \overline E : a \mapsto (s(a),s'(\varphi(a))) 
$$
is a c.p.c. split for $\overline \pi$. This shows that the middle row in (\ref{eq:pullback-resolution}) is a $\partial$-free geometric resolution for $A$.
\par For the remaining statement note that $\overline I = I \oplus I'$, which shows that $\overline I$ satisfies the UCT if $I$ and $I'$ do. Since the extension is semisplit, Lemma \ref{lem:semisplit} shows that also $\overline E$ satisfies the UCT. Separability of $\overline I$ and $\overline E$ follows by a standard two-out-of-three argument, as well.
\end{proof}

\section{The K\"unneth formula}
\label{sec:KT}

We give a revised description of the K\"unneth formula using the $\partial$-free geometric resolutions from Section \ref{sec:geores}. 
Before we start constructing the K\"unneth sequence, we make some conventions on how to treat the Tor-terms that arise. 

\begin{defn}
\label{def:tor}
	Let $G_0,G_1$ be abelian groups and let
	\[
		0 \to P_i \to Q_i \to G_i \to 0 \qquad (i=0,1)
	\]
	be a projective resolution of $G_i$. Then, the left $\mathrm{Tor}$-functor $\LTor$ and the right $\mathrm{Tor}$ functor $\RTor$ are defined by fitting into the following exact sequences:
	\begin{align*}
		0 \to \LTor(G_0,G_1) \to P_0 \otimes G_1 \to Q_0 \otimes G_1 \to G_0 \otimes G_1 \to 0,  
	\\
		0 \to \RTor(G_0,G_1) \to G_0 \otimes P_1 \to G_0 \otimes Q_1 \to G_0 \otimes G_1 \to 0 .
	\end{align*}
	The functors $\LTor$ and $\RTor$ are natural in both variables and well-defined up to natural isomorphisms (coming from different choices of projective resolutions).
\end{defn} 

While basic homological algebra tells us that $\LTor(G_0,G_1)$ is naturally isomorphic to $\RTor(G_0,G_1)$ (and therefore simply denoted by $\Tor(G_0,G_1)$), we need to fix one concrete realization of the Tor-functor for the upcoming computations. We prefer to work in the LTor-picture, but we point out that all constructions have obvious analogues using RTor instead. This would yield the same results up to isomorphism, which is made precise in Remark \ref{rem:right}. Hence, for the remainder of this paper, we use Tor and LTor synonymously.

\begin{remark}
\label{rem:semisplit}
	Throughout the rest of the paper, we will use the following fact without reference. Let 
	$$
		0 \to I \to E \to A \to 0
	$$
	be a semisplit short exact sequence. Then, for any $\cstar$-algebra $B$, the sequences 
	$$
		0 \to I \otimes B \to E \otimes B \to A \otimes B \to 0
	$$
	and 
	$$
		0 \to B \otimes  I \to B \otimes  E \to B \otimes A \to 0
	$$
	are exact, which follows from \cite[Theorem 3.2]{EH}. This remark applies in particular to $\partial$-free geometric resolutions.
\end{remark}

\begin{notation}
		Given a $\cstar$-algebra $B$ and a semisplit short exact sequence 
	\[
		\xymatrix{
			0 \ar[r] & I \ar[r] & E \ar[r]  & A \ar[r] & 0,
		}
	\]
    we obtain another short exact sequence by tensoring with $B$ from the left resp. right. The induced boundary maps are denoted by 
  		\begin{align*}
  			\partial_A \colon K_*(B \otimes A) \to K_*(B \otimes I),  \\
  			{}_A \partial \colon K_*(A \otimes B) \to K_*(I \otimes B).
  		\end{align*}
  		These maps are odd.
\end{notation}

For $\cstar$-algebras $A,B$ and for $i,j \in \Z/2\Z$, we recall the map 
$$
	\alpha_{A,B}:K_i(A)\otimes K_j(B) \to K_{i+j}(A\otimes B), $$ 
	as defined in \cite[Chapter 23]{Blackadar:KBook}. See also Section \ref{sec:proof-KT-Flip-corrected} for an explicit description.
When $A$ is in the UCT class and $K_*(A)$ is torsion-free, $\alpha_{A,B}$ is an isomorphism (by the K\"unneth formula, \cite[Theorem 23.3.1]{Blackadar:KBook}).

\begin{defn}
	\label{def:beta}
	Let $A,B$ be separable $\cstar$-algebras with $A$ in the $\T{UCT}$ class.  Given any $\partial$-free geometric resolution 
	\[
		\xymatrix{
			0 \ar[r] & I \ar[r]^\iota & E \ar[r]^\pi & A \ar[r]  & 0
		}
	\]
	of $A$ (as in Definition \ref{def:geores}), consisting of separable $\cstar$-algebras in the $\T{UCT}$ class, we define an odd map 
	\[
			\beta_{A,B} \coloneqq \alpha_{I,B}^{-1}\circ {}_A\partial \colon K_*(A \otimes B) \to \Tor(K_*(A),K_*(B)).
	\]
	Note that such a $\partial$-free geometric resolution exists by Theorem \ref{thm:existence-resolutions}.
\end{defn}

\begin{prop}
\label{prop:kunnethformula}
Let $A,B$ be separable $\cstar$-algebras with $A$ in the $\T{UCT}$ class, and let
	\[
		\xymatrix{
			0 \ar[r] & I \ar[r]^\iota & E \ar[r]^\pi & A \ar[r]  & 0
		}
	\]
	be a  $\partial$-free geometric resolution of $A$ consisting of separable $\cstar$-algebras in the $\T{UCT}$ class. This defines $\beta_{A,B}$ as above.
	Then, the range of $\beta_{A,B}$ (which is a priori contained in $K_*(I)\otimes K_*(B)$) is indeed contained in $\Tor(K_*(A),K_*(B))$ and the following diagram commutes:

\begin{equation}
	\label{eq:kunnethformula}
	\begin{aligned}
	\xymatrix{
	K_*(A \otimes B) \ar[r]^-{\beta_{A,B}} \ar[d]_{ {}_A \partial}  & \Tor(K_*(A),K_*(B)) \ar[d]^\subseteq \\
		K_*(I \otimes B) \ar[d]_{(\iota \otimes \id)_*} & K_*(I) \otimes K_*(B) \ar[d]^{\iota_* \otimes \id} \ar[l]^{\alpha_{I,B}}_\cong \\
		K_*(E \otimes B) & K_*(E) \otimes K_*(B) \ar[l]^{\alpha_{E,B}}_\cong
	}
	\end{aligned}
\end{equation}
	The map $\beta_{A,B}$ depends only on the choice of a $\partial$-free geometric resolution up to natural isomorphism coming from such a choice\footnote{But note that Tor-groups are only defined up to such a natural isomorphism.},
	and $\beta_{A,B}$ fits into the K\"unneth formula for tensor products, i.e., it makes the sequence 
	\[
		\xymatrix{
			 0 \ar[r] & K_*(A) \otimes K_*(B) \ar[r]^-{\alpha_{A,B}} & K_*(A \otimes B) \ar[r]^-{\beta_{A,B}} & \Tor(K_*(A),K_*(B)) \ar[r] & 0
		}
	\]
	exact.
\end{prop}

\begin{proof}
The top square of \eqref{eq:kunnethformula} commutes by definition of $\beta_{A,B}$ (once we show that the range of $\beta_{A,B}$ is correct).
Observe that the lower square of \eqref{eq:kunnethformula} commutes by naturality of the map $\alpha$ and that $\alpha_{I,B},\alpha_{E,B}$ are isomorphisms since $I,E$ satisfy the UCT and have free K-groups.
Since the left-hand side of the diagram is taken from the six-term exact sequence in K-theory associated to the extension $$
	0\to I\otimes B\to E\otimes B\to A\otimes B\to 0,
$$ 
one sees that $\beta_{A,B}$ does in fact map to the kernel of $\iota_*\otimes\id$, which is $\Tor(K_*(A),K_*(B))$ by definition.

To see that the definition does not depend on the particular choice of a $\partial$-free geometric resolution, let $0\to I'\to E'\to A\to 0$ be another $\partial$-free geometric resolution for $A$ consisting of separable $\cstar$-algebras in the UCT class and denote by $\beta'_{A,B}$ the correponding map described in Definition \ref{def:beta}.
In order to compare the maps $\beta_{A,B}$ and $\beta_{A,B}'$ obtained from the two different resolutions, we apply the pullback construction described in Lemma \ref{lem:comparison} to find the commutative diagram
\[
	\xymatrix{
		0\ar[r]&I\ar[r]^-\iota &E\ar[r]^-\pi&A\ar[r] &0 \\
			0 \ar[r] & \overline I  \ar[u]^q \ar[d]_{q'} \ar[r]^{\overline \iota} & \overline E \ar[u]^p \ar[d]_{p'} \ar[r]^{\overline \pi} & A \ar[r]  \ar@{=}[u] \ar@{=}[d] & 0  \\
			0\ar[r]&I'\ar[r]^-{\iota'} &E'\ar[r]^-{\pi'}&A\ar[r]&0
	}
\]
with exact rows. By Lemma \ref{lem:comparison}, the middle row is again a $\partial$-free geometric resolution such that $\overline I$ and $\overline E$ satisfy the UCT and  therefore satisfies the hypothesis of Definition \ref{def:beta} and also yields a map $\overline{\beta}_{A,B}$. Using the definitions for the different $\beta$'s, we find the diagram 
\[
\resizebox{\textwidth}{!}{
	\xymatrix{
		& K_{*+1}(A\otimes B) \ar@{=}[rr] \ar[dl]_-{\beta_{A,B}} \ar[dd]_(.3){{}_A \partial} && K_{*+1}(A\otimes B) \ar[dl]_-{\overline{\beta}_{A,B}} \ar[dd]_(.3){{}_A \partial}\\
		\Tor(K_*(A),K_*(B)) \ar[dd]_(.7)\subseteq && \ar[ll]_(.4)\cong \Tor(K_*(A),K_*(B)) \ar[dd]_(.7)\subseteq\\
		& K_*(I \otimes B) && K_*(\overline I \otimes B) \ar[ll]_(.7){(q \otimes \id)_*} \\
		K_*(I) \otimes K_*(B)  \ar[ur]^\cong_{\alpha_{I,B}} && K_*(\overline I) \otimes K_*(B) \ar[ll]_{q_* \otimes \id} \ar[ur]^\cong_{\alpha_{\overline I, B}}
		}
	}
\]
We need to show that the top face of the cube commutes.  This follows if all other faces commute. The left and right faces commute by definition of $\beta_{A,B}$ and $\overline \beta_{A,B}$, the face on the back commutes by naturality of the boundary maps ${}_A \partial$. The bottom face commutes by naturality of $\alpha$ and the front commutes by naturality of the Tor-functor. Hence $\beta_{A,B}$ agrees with $\overline{\beta}_{A,B}$ up to the natural isomorphism coming from the different choices of projective resolutions for $K_*(A)$. 
Since Tor-groups are only defined up to such isomorphisms, we have  $\beta_{A,B}=\overline{\beta}_{A,B}$.
The same argument applies to the second geometric realization, so that  $\beta_{A,B}'=\overline{\beta}_{A,B}=\beta_{A,B}$.

Finally, the long exact sequence
\[
	\xymatrix{
		  K_*(I\otimes B) \ar[r]^{(\iota\otimes\id)_*} & K_*(E\otimes B) \ar[r]^{(p\otimes\id)_*} & K_*(A\otimes B) \ar[r]^{{}_A\partial} &  K_*(I\otimes B) \ar[d] \\ 
		\vdots \ar[u] & & & \vdots
	}
\]
unsplices to
\[
	\xymatrix{
		0 \ar[r] & \mathrm{Coker}((\iota\otimes\id)_*) \ar[r] & K_*(A\otimes B) \ar[r] &  \mathrm{Ker}((\iota\otimes\id)_*) \ar[r] & 0.
	}
\]
Now, since $\alpha_{I,B}$ and $\alpha_{E,B}$ are isomorphisms, one finds
\[
	 \mathrm{Coker}((\iota\otimes\id)_*) \cong \mathrm{Coker}(\iota_*\otimes\id) \cong K_*(A)\otimes K_*(B)
\]
and
\[
	\mathrm{Ker}((\iota\otimes\id)_*) \cong \mathrm{Ker}(\iota_*\otimes\id) \cong \Tor(K_*(A),K_*(B)),
\]
which, after checking that the maps match up, establishes the K\"unneth formula for tensor products.
\end{proof}

\begin{remark}
\label{rem:natural}
The map $\beta_{A,B}$ constructed in Definition \ref{def:beta} is natural in both variables.
While naturality in the second variable is straightforward, naturality in the first variable can been shown using the pullback construction of Lemma \ref{lem:comparison}. 

\DE{we sketch the proof in the first variable here: given a $\Star$-homomorphism $\varphi\colon A\to A'$, one uses the pullback construction of Lemma \ref{lem:comparison} together with $\beta_{A,B}=\overline{\beta}_{A,B}$ (by \ref{prop:kunnethformula}) to see that
\[
	\begin{array}{rl}
		Tor(\varphi)\circ\beta_{A,B} &  = Tor(\varphi)\circ\overline{\beta}_{A,B} \\
		& =(q'_*\otimes\id)\circ\overline{\beta}_{A,B} \\
		& =(q'_*\otimes\id)\circ\alpha^{-1}_{\overline{I},B}\circ {}_A\partial \\
		& =\alpha^{-1}_{I',B}\circ(q'_*\otimes\id)\circ {}_A\partial \\
		& =\alpha^{-1}_{I',B}\circ {}_A\partial \circ (\varphi\otimes\id)_* \\
		& = \beta'_{A,B} \circ(\varphi\otimes\id)_* .
	\end{array}
\]
}	
\end{remark}

\begin{remark}
\label{rem:right}
Similar to Definition \ref{def:beta}, we could use the right Tor-functor (see Definition \ref{def:tor}) and define 
    	$$
		\beta_{A,B}^R \colon K_*(A \otimes B) \to \RTor(K_*(A),K_*(B))
	$$
by $x \mapsto \alpha_{A,J}^{-1} \circ \partial_B(x)$, whenever $B$ is a separable  $\cstar$-algebra satisfying the $\T{UCT}$. 
Then, for any $\partial$-free geometric resolution of the second variable
\[
	\xymatrix{
		0 \ar[r] & J \ar[r]^\iota & F \ar[r]^\pi & B \ar[r]  & 0,
	}
\]
consisting of separable  $\cstar$-algebras satisfying the $\T{UCT}$, the map $\beta_{A,B}^R$ makes the diagram
\[
    	\xymatrix{
		K_*(A \otimes B) \ar@{..>}[r]^-{\beta_{A,B}^R} \ar[d]_{ \partial_B}  & \RTor(K_*(A),K_*(B)) \ar[d]^\subseteq \\
		K_*(A \otimes J) \ar[d]_{(\id \otimes \iota)_*} & K_*(A) \otimes K_*(J) \ar[d]^{\id \otimes \iota_*} \ar[l]^{\alpha_{A,J}}_\cong \\
		K_*(A \otimes F) & K_*(A) \otimes K_*(F) \ar[l]^{\alpha_{A,F}}_\cong
	}
\]
commute and yields the K\"unneth formula for tensor products just as in Proposition \ref{prop:kunnethformula}.
Again, $\beta_{A,B}^R$ is natural in both variables and coincides with $\beta_{A,B}$ (from Definition \ref{def:beta})	 up to natural isomorphism.
\end{remark}

\section{The flip map and the K\"unneth formula}

\begin{lemma}
	\label{lem:sum}
	Consider two $\partial$-free geometric resolutions
	\[
		\xymatrix@R-1.5em{
			0 \ar[r] & I \ar[r]^-{\iota_A} & E \ar[r]^-{\pi_A} & A \ar[r] & 0 , \\ 
			0 \ar[r] & J \ar[r]^-{\iota_B} & F \ar[r]^-{\pi_B} & B \ar[r] & 0 
		} 
	\]
	consisting of separable $\cstar$-algebras satisfying the $\T{UCT}$. Then, the K-theory six-term exact sequences associated to
	\[
		\xymatrix@R-1.5em{
			0 \ar[r] & I \otimes F \ar[r]^-{\iota_1} & I \otimes F + E \otimes J \ar[r] & A \otimes J \ar[r] & 0,   \\ 
			0 \ar[r] & E \otimes J \ar[r]^-{\iota_2} & I \otimes F + E \otimes J \ar[r] & I \otimes B \ar[r] & 0 
		}
	\]
	degenerate into two short exact sequences
	\[
		\xymatrix@R-1.5em{
			0 \ar[r] & K_*(I \otimes F) \ar[r] & K_*(I \otimes F + E \otimes J) \ar[r] & K_*(A \otimes J) \ar[r]  &0, \\
				0 \ar[r] & K_*(E \otimes J) \ar[r] & K_*(I \otimes F + E \otimes J) \ar[r] & K_*(I \otimes B) \ar[r]  &0.
		}
	\]
\end{lemma}

\begin{proof}
	We consider the following commutative diagram: 
	
	\begin{equation}	\label{eq sum}
		\begin{gathered}		
		\xymatrix{
			0 \ar[r] & I \otimes F \ar@{=}[d] \ar[r]^-{\iota_1} & I \otimes F + E \otimes J \ar[d]^{\iota} \ar[r] & A \otimes J \ar[d]^-{\id \otimes \iota_B} \ar[r] & 0  \\
			0 \ar[r] & I \otimes F \ar[r] & E \otimes F \ar[r]^-{\pi_A \otimes \id} & A \otimes F \ar[r] & 0
		}
		\end{gathered}
		\end{equation}
		
	The bottom row  of (\ref{eq sum}) induces the following six-term sequence
	\[
		\xymatrix{
			K_0(I \otimes F) \ar[r] & K_0(E \otimes F) \ar[r] & K_0(A \otimes F) \ar[d]^0 \\
			K_1(A \otimes F) \ar[u]^0 & K_1(E \otimes F) \ar[l] & K_1(I \otimes F) \ar[l]
		}
	\]
	Indeed, by using naturality of the K\"unneth formula and using the fact that $K_*(F)$ is free, one sees that the maps $K_i(I \otimes F) \to K_i(E \otimes F)$ are injective. Now, by naturality of the boundary maps in the six-term sequence we see that the following diagram commutes:
	\[
		\xymatrix{
			K_{*+1}(A \otimes J) \ar[r]^-{(\id \otimes \iota_B)_*} \ar[d]^\partial & K_{*+1}(A \otimes F) \ar[d]^0 \\
			K_*(I \otimes F) \ar@{=}[r] & K_*(I \otimes F)
		}
	\]
	It follows that the boundary maps associated to the top row  of (\ref{eq sum}) are zero.
\end{proof}

\begin{lemma}
	\label{lem:alg}
	Consider two $\partial$-free geometric resolutions 
	\begin{align*}
		\xymatrix@R-1.5em{
			0 \ar[r] & I \ar[r]^-{\iota_A} & E \ar[r]^-{\pi_A} & A \ar[r] & 0 , \\ 
			0 \ar[r] & J \ar[r]^-{\iota_B} & F \ar[r]^-{\pi_B} & B \ar[r] & 0 
		} 
	\end{align*}
	consisting of separable $\cstar$-algebras satisfying the $\T{UCT}$.  This produces the following commutative diagram with exact rows and columns:
	\[
		\xymatrix{
			& 0 \ar[d] & 0 \ar[d] & K_*(A \otimes B) \ar[d]^-{{}_A \partial} \ar@/_4.2em/@{-->}[dddll]_-{\theta_B} & \\
			0 \ar[r] & K_*(I \otimes J) \ar[r] \ar[d] & K_*(I \otimes F) \ar[r]^-{(\id \otimes \pi_B)_*} \ar[d]^-{(\iota_A \otimes \id)_*} & K_*(I \otimes B) \ar[r] \ar[d] & 0 \\
			0 \ar[r] & K_*(E \otimes J) \ar[r]^-{(\id \otimes \iota_B)_*} \ar[d]_-{(\pi_A \otimes \id)_*} & K_*(E \otimes F) \ar[r] \ar[d] & K_*(E \otimes B) \ar[r] \ar[d] & 0 \\
		K_*(A \otimes B) \ar[r]^-{\partial_B} &	K_*(A \otimes J) \ar[r] \ar[d] & K_*(A \otimes F) \ar[r] \ar[d] & K_*(A \otimes B) \\
		& 	0 & 0 
		}
	\]
	Then, the odd homomorphism
	$$
		\theta_B \colon K_*(A \otimes B) \to K_*(A \otimes J),
	$$ 
	induced by this diagram (as in \cite[Section 3]{aifK}), satisfies
	$$
		\theta_B = - \partial_B.
	$$
\end{lemma}

\begin{proof}
	We first recall the construction of  $\theta_B$. Fix $x \in K_{i+1}(A \otimes B)$ and find some lift $a \in K_i(I \otimes F)$ with $(\id \otimes \pi_B)_i(a) = {}_A \partial(x)$. By exactness of the third row at $K_*(E \otimes F)$ we can find (a unique) $b \in K_i(E \otimes J)$ with $(\id \otimes \iota_B)_i(b) = (\iota_A \otimes \id)_i(a)$. Then one defines $\theta_B(x) = (\pi_A \otimes \id)_i(b)$. One can check that the outcome does not depend on the choice of $a$ and that $\theta_B$ is a group homomorphism.
	\par To prove the lemma we consider the following commutative diagram:
	\[
		\xymatrix{
			K_*(A \otimes B) \ar[r]^-{\partial_B} \ar@{=}[d] & K_*(A \otimes J) & K_*(E \otimes J)\ar@{^{(}-->}[dl]_-{(\iota_2)_*} \ar@{->>}[l]_-{(\pi_A \otimes \id)_*} \ar[d]^-{(\id \otimes \iota_B)_*} \\
			K_*(A \otimes B) \ar@{=}[d] \ar[r]^-{\partial} & K_*(I \otimes F + E \otimes J) \ar@{..>>}[u] \ar@{-->>}[d] \ar[r]^-{\iota_*} & K_*(E \otimes F)  \\
			K_*(A \otimes B) \ar[r]_-{{}_A \partial} & K_*( I \otimes B) & K_*(I \otimes F) \ar@{->>}[l]^-{(\id \otimes \pi_B)_*} \ar@{_{(}..>}[ul]^-{(\iota_1)_*} \ar[u]_-{(\iota_A \otimes \id)_*}
		}
	\]
	(The dotted and dashed arrows indicate exact sequences in the diagram, as we shall explain.)
	The middle row in this diagram is induced by the six-term exact sequence associated to 
	\[
		\xymatrix{
			0 \ar[r] &   I \otimes F + E \otimes J \ar[r]^-{\iota} &  E \otimes F \ar[r]^-{\pi_A \otimes \pi_B} &  A \otimes B \ar[r] &  0
		} .
	\]
	The dotted and dashed ways come from the six-term exact sequence associated to the following short exact sequences (respectively)
	\begin{align*}
		\xymatrix{
			0 \ar[r] & I \otimes F \ar[r]^-{\iota_1} & I \otimes F + E \otimes J \ar[r] & A \otimes J \ar[r]  & 0
		},
		\\
		\xymatrix{
			 0 \ar[r] & E \otimes J \ar[r]^-{\iota_2} & I \otimes F + E \otimes J \ar[r] & I \otimes B \ar[r] & 0 .
		}
	\end{align*}
	As shown in Lemma \ref{lem:sum}, their associated six-term sequences degenerate to two short exact sequences. The left part of the diagram is induced by naturality of boundary maps, i.e., by considering the following maps of short exact sequences:
	\[
		\xymatrix{
			0 \ar[r] & A \otimes J \ar[r] & A \otimes F \ar[r] & A \otimes B \ar[r] & 0 \\
			0 \ar[r] & I \otimes F + E \otimes J \ar[u]  \ar[d] \ar[r]^-\iota & E \otimes F  \ar[d] \ar[u] \ar[r] & A \otimes B \ar[r] \ar@{=}[u] \ar@{=}[d] & 0 \\
			0 \ar[r] & I \otimes B \ar[r] & E \otimes B \ar[r] & A \otimes B \ar[r] & 0 
		}
	\]
	\par 
	Now, fix  $x \in K_*(A \otimes B)$ and define $z \coloneqq \partial (x) \in K_*( I \otimes F + E \otimes J)$. Choose $a \in K_*(I \otimes F)$ such that $(\id \otimes \pi_B)_*(a) =  {}_A \partial(x)$. We see that  $z-(\iota_1)_*(a)$ gets killed by the map $K_*(I \otimes F + E \otimes J) \to K_*(I \otimes B)$. By exactness of the dashed way, we see that  $z-(\iota_1)_*(a) \in K_*(E \otimes J)$. Hence, there exists some $b \in K_*(E \otimes J)$ with $z = (\iota_1)_*(a)+(\iota_2)_*(b)$.  It follows that  
	\begin{align*}
		 0 & = \iota_*(\partial(x))  = \iota_*(z) = \iota_*((\iota_1)_*(a)+(\iota_2)_*(b)) \\
		 & = (\iota_A \otimes \id)_*(a) + (\id \otimes \iota_B)_*(b).
	\end{align*} 
	Since now  $(\iota_A \otimes \id)_*(a) =  (\id \otimes \iota_B)_*(-b)$, we get by definition of $\theta_B$  that $\theta_B(x) = (\pi_A \otimes \id)_*(-b) = - \partial_B(x)$.
\end{proof}

\section{Proof of Lemma \ref{lem:KunnethFlipCorrected}}
\label{sec:proof-KT-Flip-corrected}

\begin{proof}[Proof that diagram \eqref{eq:KunnethFlipDiag1} commutes] By naturality of the maps involved and by unitizing if necessary, we may assume that $A$ and $B$ are both unital, so that $K_0(A)$ and $K_0(B)$ are generated by classes of projections in matrix algebras over $A$ and $B$ respectively.
Let $p$ and $q$ be  projections in some matrix algebra over $A$ resp. $B$ and fix unitaries $u$ and $v$ in some matrix algebra over $A$ resp.  $B$. Then, the following holds (where $1$ denotes the identity in the respective matrix algebras):
\[
	\begin{array}{l}
		\alpha_{A,B}([p]_0 \otimes [q]_0) = [p \otimes q]_0, \\
		\alpha_{A,B}([p]_0 \otimes [v]_1) = [p \otimes v + (1-p) \otimes 1]_1, \\
		\alpha_{A,B}([u]_1 \otimes [q]_0) = [u \otimes q + 1 \otimes (1-q)]_1.
	\end{array}
\]
Having these formulas, one checks by hand that they interact with the flip as claimed in Lemma \ref{lem:KunnethFlipCorrected}. For the remaining summand we use that $\alpha$ is natural. Let $\phi_A \colon C(\mathbb T) \to A : z \mapsto u$ and $\phi_B \colon C(\mathbb T) \to B :  z \mapsto v$, where $z$ is the canonical generator of $C(\mathbb T)$. Let us consider the following diagram
\[		
	\resizebox{\textwidth}{!}{
	\xymatrix{
	& K_1(C(\mathbb T)) \otimes K_1(C(\mathbb T))  \ar[dd]|-(.2){-\sigma_{K_1(C(\mathbb T)),K_1(C(\mathbb T))}} \ar[dl]_-{(\phi_A)_1 \otimes (\phi_B)_1}  \ar[rr]^(.4){\alpha_{C(\mathbb T),C(\mathbb T)}}  & & K_0(C(\mathbb T) \otimes C(\mathbb T))  \ar[dl]_-{(\phi_A \otimes \phi_B)_1} \ar[dd]|-(.2){(\sigma_{C(\mathbb T),C(\mathbb T)})_0} \\
	K_1(A) \otimes K_1(B) \ar[rr]_(.4){\alpha_{A,B}} \ar[dd]|-(.2){-\sigma_{K_1(A),K_1(B)}}  & & K_0(A \otimes B) \ar[dd]|-(.2){(\sigma_{A,B})_0} \\
	& K_1(C(\mathbb T)) \otimes K_1(C(\mathbb T)) \ar[dl]^{(\phi_B)_1 \otimes (\phi_A)_1} \ar[rr]_(.4){\alpha_{C(\mathbb T),C(\mathbb T)}} & & K_0(C(\mathbb T) \otimes C(\mathbb T)) \ar[dl]^{(\phi_B \otimes \phi_A)_0} \\
	K_1(B) \otimes K_1(A) \ar[rr]_(.4){\alpha_{B,A}} & & K_0(B \otimes A)
	}
	}
\]
The top and bottom faces commute by naturality of $\alpha$. The back face commutes by Example \ref{ex:bott-element}. The left and right faces commute by definition of the flip map. A diagram chase now shows that 
$$
	(\sigma_{A,B})_0(\alpha_{A,B}([u]_1 \otimes [v]_1)) = - \alpha_{B,A}([v]_1 \otimes [u]_1).
$$
\end{proof}

\begin{proof}[Proof that diagram \eqref{eq:KunnethFlipDiag2} commutes]
Fix $\partial$-free geometric resolutions for $A$ and $B$:
	\begin{align*}
		\xymatrix{
			0 \ar[r] & I \ar[r]^-{\iota_A} & E \ar[r]^-{\pi_A} & A \ar[r] & 0 ,
		}  \\ 
		\xymatrix{
			0 \ar[r] & J \ar[r]^-{\iota_B} & F \ar[r]^-{\pi_B} & B \ar[r] & 0 ,
		} 
	\end{align*}
	consisting of separable $\cstar$-algebras satisfying the UCT. For the definition of the algebraic isomorphism
	 $$
	 	\eta_{K_i(A),K_j(B)} \colon \Tor(K_i(A),K_j(B)) \to \Tor(K_j(B),K_i(A)),
	 $$ one considers the following commutative diagram and performs a diagram chase:
	 
	 \begin{landscape}
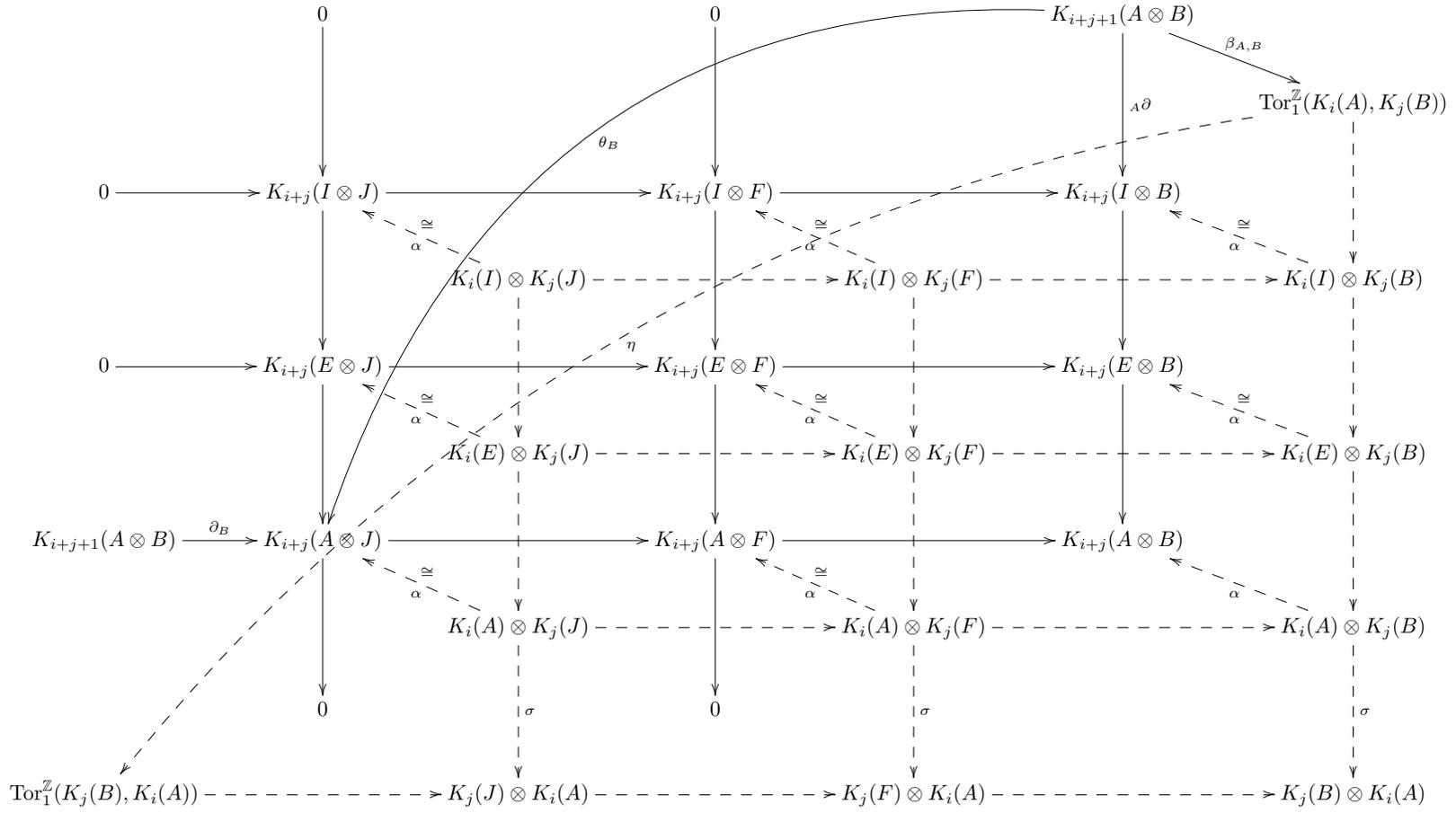
	 \[
	\resizebox{1.7\textwidth}{!}{
	 	\xymatrix{
	 		& 0 \ar[dd] & & 0  \ar[dd]  & & K_{i+j+1}(A \otimes B) \ar@/_9em/@{->}[ddddddllll]^-{\theta_B} \ar[dd]^-{{}_A \partial} \ar[dr]^-{\beta_{A,B}} & \\ 
	 	  	& & & & & &  \ar@/_7em/@{-->}[ddddddddllllll]^-{\eta}\Tor(K_i(A),K_j(B)) \ar@{-->}[dd]  \\
	 		0 \ar[r] & K_{i+j}(I \otimes J) \ar[dd] \ar[rr] & & K_{i+j}(I \otimes F) \ar[dd] \ar[rr] & & K_{i+j}(I \otimes B) \ar[dd] \\
	 		& & K_i(I) \otimes K_j(J) \ar@{-->}[ul]^-{\alpha}_-\cong \ar@{-->}[rr] \ar@{-->}[dd] & & K_i(I) \otimes K_j(F) \ar@{-->}[ul]^-\alpha_-\cong \ar@{-->}[rr] \ar@{-->}[dd] & & K_i(I) \otimes K_j(B) \ar@{-->}[ul]^-\alpha_-\cong \ar@{-->}[dd] \\
	 		0 \ar[r]	& K_{i+j}(E \otimes J) \ar[rr] \ar[dd] & & K_{i+j}(E \otimes F)  \ar[rr] \ar[dd]  &  & K_{i+j}(E \otimes B) \ar[dd]  \\
	 		& & K_i(E) \otimes K_j(J) \ar@{-->}[dd] \ar@{-->}[rr] \ar@{-->}[ul]^-{\alpha}_-\cong & & K_i(E) \otimes K_j(F) \ar@{-->}[rr] \ar@{-->}[dd]  \ar@{-->}[ul]^-{\alpha}_-\cong & & K_i(E) \otimes K_j(B)  \ar@{-->}[dd] \ar@{-->}[ul]^-{\alpha}_-\cong \\
	 	K_{i+j+1}(A \otimes B) \ar[r]^-{\partial_B}	& K_{i+j}(A \otimes J) \ar[dd]  \ar[rr] && K_{i+j}(A \otimes F) \ar[dd] \ar[rr] & & K_{i+j}(A \otimes B) \\
	 		& & K_i(A) \otimes K_j(J)  \ar@{-->}[rr] \ar@{-->}[dd]^-\sigma \ar@{-->}[ul]^-{\alpha}_-\cong && K_i(A) \otimes K_j(F) \ar@{-->}[dd]^-\sigma  \ar@{-->}[rr] \ar@{-->}[ul]^-{\alpha}_-\cong && K_i(A) \otimes K_j(B) \ar@{-->}[ul]^-{\alpha} \ar@{-->}[dd]^-\sigma \\ 
	 		&  0 & & 0 
	 		\\ 
	 	\Tor(K_j(B),K_i(A)) \ar@{-->}[rr]	& & K_j(J) \otimes K_i(A)  \ar@{-->}[rr] & & K_j(F) \otimes K_i(A) \ar@{-->}[rr] & & K_j(B) \otimes K_i(A) 
	 	}
	 }
	\]
	
	\end{landscape}
	 	 
	The front (dashed) layer is induced by forming the double complex associated to the projective resolutions of $K_i(A)$ and $K_j(B)$ coming from our $\partial$-free geometric resolutions. By  naturality of $\alpha$ one arrives at the (back) solid layer, which is the one described in Lemma \ref{lem:alg}. We start with $x \in K_{i+j+1}(A \otimes B)$ and remember that 
	$$
		\beta_{A,B} \colon K_{i+j+1}(A \otimes B) \to \Tor(K_i(A),K_j(B))
	$$
	is the map $\beta_{A,B}$ from Definition \ref{def:beta} (followed by the coordinate projection onto $\Tor(K_i(A),K_j(B))$). Then $\eta_{K_i(A),K_j(B)}(\beta_{A,B}(x))$ is computed by performing a diagram chase in the front (dashed) layer. By commutativity of the above diagram, we see that 
	 $$
	 	\eta_{K_i(A),K_j(B)}(\beta_{A,B}(x)) = \sigma_{K_i(A),K_j(J)} \circ \alpha_{A,J}^{-1} \circ \theta_B(x),
	 $$
	 where $\theta_B$ is as in Lemma \ref{lem:alg}. By the same lemma we know that $\theta_B(x) = -\partial_B(x)$. We thus get 
	 \[
	 	\eta_{K_i(A),K_j(B)}(\beta_{A,B}(x)) = -  \sigma_{K_i(A),K_j(J)} \circ \alpha_{A,J}^{-1} \circ \partial_B(x).
	 \]
	Consider the following commuting diagram:
	
	\[
		\xymatrix{
			K_{i+j+1}(A \otimes B) \ar[r]^{\partial_B} \ar[d]_{K_{i+j+1}(\sigma_{A,B})} & K_{i+j}(A \otimes J) \\
			K_{i+j+1}(B \otimes A) \ar[r]^{{}_B \partial} \ar[d]_{\beta_{B,A}} & K_{i+j}(J \otimes A) \ar[u]_{K_{i+j}(\sigma_{J,A})}  \\
			\Tor(K_j(B),K_i(A)) \ar[r]^-{\subseteq} & K_j(J) \otimes K_i(A) \ar[u]_{\alpha_{J,A}} 
		}
	\] 
	By walking along the outer square we may replace $\partial_B$ and arrive at the following:
	\begin{eqnarray*}
		&& \hspace*{-3em} \eta_{K_i(A),K_j(B)} \circ \beta_{A,B}(x) \\ 
		&=& - \sigma_{K_i(A),K_j(J)} \circ \alpha_{A,J}^{-1} \circ K_{i+j}(\sigma_{J,A}) \circ \alpha_{J,A}  \circ \beta_{B,A} \circ K_{i+j+1}(\sigma_{A,B})(x) \\
		&\stackrel{\eqref{eq:KunnethFlipDiag1}}=& -(-1)^{ij}\beta_{B,A}\circ K_{i+j+1}(\sigma_{A,B})(x),
	\end{eqnarray*}
	as required.
\end{proof}

\begin{remark}
    The conclusion of Lemma \ref{lem:KunnethFlipCorrected} remains the same if we consider $\beta_{A,B}^R$ instead. Indeed, the maps $\beta_{A,B}^L$ and $\beta_{A,B}^R$ are related by a minus sign:
    \[
        \xymatrix{
            K_{i+j+1}(A \otimes B) \ar[r]^-{\beta_{A,B}^L} \ar[d]^{-1} & \LTor(K_i(A),K_j(B)) \ar[d]^\cong  \\
            K_{i+j+1}(A \otimes B) \ar[r]_-{\beta_{A,B}^R} & \RTor(K_i(A),K_j(B))
        }
    \]
    The vertical isomorphism is the algebraic one induced by diagram chasing. Then the following diagram commutes: 
    \[
        \xymatrix{
             K_*(A \otimes B)  \ar@/_7em/[ddd]_{(\sigma_{A,B})_*} \ar[r]^-{\beta_{A,B}^R} \ar[d]^{-1} & \RTor(K_*(A),K_*(B)) \ar[d]^\cong \ar@/^7em/[ddd]^{\eta_R}  \\
            K_*(A \otimes B) \ar[r]_-{\beta_{A,B}^L} \ar[d]_{(\sigma_{A,B})_*} & \LTor(K_*(A),K_*(B)) \ar[d]^{ \eta} \\
             K_*(B \otimes A) \ar[r]^-{\beta_{B,A}^L} \ar[d]^{-1} & \LTor(K_*(B),K_*(A)) \ar[d]^\cong  \\
            K_*(B \otimes A) \ar[r]_-{\beta_{B,A}^R} & \RTor(K_*(B),K_*(A))
        }
    \]
\end{remark}

\section{Proof of Theorem \ref{thm:MainThmCorrected}}
\label{sec:MainThmCorrected}

\begin{proof}[{Proof of Theorem \ref{thm:KKsuff}}]
In all five cases, we see first using the K\"unneth formula that
	\begin{align*} 	
		\alpha_{A,A} \colon &K_0(A)\otimes K_0(A) \to K_0(A\otimes A) \quad \text{and} \\
			 \beta_{A,A} \colon & K_1(A\otimes A) \to \Tor(K_1(A),K_1(A)) 
		\end{align*}
are isomorphisms.
Second, we see that both
	\begin{align*}
		\sigma_{K_0(A),K_0(A)} \colon &K_0(A)\otimes K_0(A) \to K_0(A)\otimes K_0(A)
\quad \text{and}\quad  \\
\eta_{K_1(A),K_1(A)}  \colon &K_1(A)\otimes K_1(A) \to K_1(A)\otimes K_1(A)
\end{align*}
are the respective identity maps (this uses \cite[Proposition 5.1]{aifK}).

Thus by Lemma \ref{lem:KunnethFlipCorrected},
\begin{align*}
K_0(\sigma_{A,A}) = (-1)^{0\cdot 0}\alpha_{A,A}\circ \sigma_{K_0(A),K_0(A)}\circ \alpha_{A,A}^{-1} = \id_{K_0(A\otimes A)}
\end{align*}
and
\begin{align*}
K_1(\sigma_{A,A}) = (-1)^{1+1\cdot 1}\beta_{A,A}^{-1}\circ \eta_{K_1(A),K_1(A)}\circ \beta_{A,A} = \id_{K_1(A\otimes A)}.
\end{align*}

Next (as in the proof of \cite[Theorem 5.2]{aifK})), we have 
$$
	\mathrm{Ext}_1^{\Z}(K_i(A\otimes A),K_{1-i}(A\otimes A))=0 \qquad (i=0,1),
$$ by using \cite[Lemma 1.1]{aifK} in case (ii) \AS{This was case (v) before}. So by the UCT, it follows that $\sigma_{A,A}$ agrees with $\id_{A\otimes A}$ in KK.
\end{proof}

\begin{proof}[{Proof of Corollary \ref{cor:Suff}}]
The proof of \cite[Corollary 5.3]{aifK} goes through unchanged, except using Theorem \ref{thm:KKsuff} in place of \cite[Theorem 5.2]{aifK}.
\end{proof}

To prove Theorem \ref{thm:NecessaryCorrected}, we follow roughly the same argument as in \cite[Section 6]{aifK}, but need to work with the grading.
Thus we will prove variants of the intermediary lemmas from \cite[Section 6]{aifK}, where the main change is to work with direct summands of $K_i(A)$ rather than of $K_*(A)=K_0(A)\oplus K_1(A)$.

\begin{lemma}[{cf.\ \cite[Lemma 6.2]{aifK}}]
\label{lem:BasicRestrictions}
Let $A$ be a separable $\cstar$-algebra in the $\T{UCT}$ class with approximately inner flip.
Suppose that $G_i$ is a direct summand of $K_0(A)$ or $K_1(A)$, for $i=1,2$, with $G_1\cap G_2 = \{0\}$.
Then:
\begin{enumerate}
\item $G_1 \otimes G_2 = 0$; and
\item $\Tor(G_1,G_2)=0$.
\end{enumerate}
\end{lemma}

\begin{proof}
Since $\sigma_{A,A}:A\otimes A \to A \otimes A$ is approximately inner, $K_i(\sigma_{A,A})$ must be the identity map on $K_i(A\otimes A)$, for $i=0,1$.
Let $G_i$ be a direct summand of $K_{t_i}(A)$, for $i=1,2$.
If $t_1\neq t_2$, then it follows by Lemma \ref{lem:KunnethFlipCorrected} that $\sigma_{K_{t_1}(A),K_{t_2}(A)}=0$ and $\eta_{K_{t_1}(A),K_{t_2}(A)}=0$, which implies $K_{t_1}(A)\otimes K_{t_2}(A)=0$ and $\Tor(K_{t_1}(A),K_{t_2}(A))=0$.

On the other hand, if $t_1=t_2=t$ then by Lemma \ref{lem:KunnethFlipCorrected}, 
\[ \sigma_{K_{t}(A),K_{t}(A)} = \pm\id_{K_{t}(A)\otimes K_{t}(A)}, \]
whereas the computation in \cite[Lemma 6.2]{aifK} shows that if $G_1\otimes G_2\neq 0$ then $\sigma_{K_t(A),K_t(A)}$ cannot be $\pm\id_{K_t(A)\otimes K_t(A)}$.

Likewise, by Lemma \ref{lem:KunnethFlipCorrected},
\[ \eta_{K_{t}(A),K_{t}(A)} = \pm\id_{\Tor(K_{t}(A),K_{t}(A))}, \]
whereas the computation in \cite[Lemma 6.2]{aifK} shows that if $\Tor(G_1,G_2)\neq 0$ then $\eta_{K_t(A),K_t(A)}$ cannot be $\pm\id_{\Tor(K_t(A),K_t(A))}$.
\end{proof}

The following is the same as \cite[Lemma 6.3]{aifK}. The proof there works using Lemma \ref{lem:BasicRestrictions} in place of \cite[Lemma 6.2]{aifK}.

\begin{lemma}
\label{lem:TensorKComp}
Let $A$ be a separable $\cstar$-algebra in the $\T{UCT}$ class which has approximately inner flip.
Then 
\begin{align*}
K_0(A \otimes A) &\cong \big(K_0(A) \otimes K_0(A)\big) \oplus \big(K_1(A) \otimes K_1(A)\big) \quad \text{and} \\
K_1(A \otimes A) &\cong \Tor(K_0(A),K_0(A)) \oplus \Tor(K_1(A),K_1(A)).
\end{align*}
\end{lemma}

\begin{lemma}[{cf.~\cite[Lemma 6.4]{aifK}}]
\label{lem:NoCycSummand}
Let $A$ be a separable $\cstar$-algebra in the $\T{UCT}$ class, which has approximately inner flip and let $G_p$ be a direct summand of $K_0(A)$ or $K_1(A)$ which is a nonzero $p$-group for some prime $p$.
Then $G_p \cong \Q_{p^\infty}/\Z$.
\end{lemma}

\begin{proof}
The proof of \cite[Lemma 6.4 (i)]{aifK} can be used with Lemma \ref{lem:BasicRestrictions} in place of \cite[Lemma 6.2]{aifK}.
Let us explain why the strengthened hypotheses of Lemma \ref{lem:BasicRestrictions} still apply.

First, \cite[Lemma 6.2 (i)]{aifK} is invoked in the first paragraph of the proof of \cite[Lemma 6.4 (i)]{aifK}, using two direct summands of $G_p$ as $G_1$ and $G_2$.
Since $G_p$ is assumed to be a direct summand of either $K_0(A)$ or $K_1(A)$, it follows that these two direct summands of $G_p$ have the same property, so Lemma \ref{lem:BasicRestrictions} can be used here.

Second, \cite[Lemma 6.2 (ii)]{aifK} is invoked in the third paragraph of the proof of \cite[Lemma 6.4 (i)]{aifK}, using $A\otimes A$ in place of $A$ and using two copies of $\Z/n\Z$ as $G_1$ and $G_2$; one copy sits inside $K_0(A\otimes A)$ and the other inside $K_1(A\otimes A)$.
Thus, the required hypothesis of Lemma \ref{lem:BasicRestrictions} does apply to these $G_1$ and $G_2$, as required.
\end{proof}

\begin{proof}[Proof of Theorem \ref{thm:NecessaryCorrected}]
Set $G_i \coloneqq K_i(A)$ for $i=0,1$, and $G \coloneqq K_*(A)=G_0\oplus G_1$.
Let $T_{G_i}$ denote the torsion subgroup of $G_i$ for $i=0,1$, so $T_G \coloneqq T_{G_0}\oplus T_{G_1}$ is the torsion subgroup of $G$.

We follow the following three steps (similar to the proof of \cite[Theorem 6.1]{aifK}):
\begin{enumerate}
\item $G/T_G$ has rank at most one;
\item $T_{G_i} \cong \Q_{m_i}/\Z$, for some supernatural number $m_i$ of infinite type; and
\item the theorem.
\end{enumerate}

For step (i), if $G/T_G$ has rank greater than one, then $K_*(A\otimes \mathcal Q) \cong G\otimes \Q$ and so there would be a direct summand (a subspace) $G_i$ of $K_0(A\otimes \mathcal Q)$ or $K_1(A\otimes \mathcal Q)$ for $i=1,2$, such that $G_1 \cap G_2 = \{0\}$.
This would contradict Lemma \ref{lem:BasicRestrictions} (i).

For step (ii), fix $i\in\{0,1\}$.
We may write $T_{G_i}$ as a direct sum of $p$-components $T_p$, over all primes $p$.
Fix a prime $p$; it will suffice to show that $T_p$ is either $0$ or $\Q_{p^\infty}/\Z$.

We now consider the algebra $A \otimes \mathcal F_{1,p^\infty}$ (where we recall that $K_*(\mathcal F_{1,p^\infty}) \cong 0 \oplus \Q_{p^\infty}/\Z$).
Since $K_0(\mathcal F_{1,p^\infty})=0$, the K\"unneth formula gives the following exact sequence:
\begin{equation}
\label{eq:Necessary1}
K_{1-i}(A)\otimes K_1(\mathcal F_{1,p^\infty}) \to K_i(A\otimes \mathcal F_{1,p^\infty}) \to \Tor(K_i(A),K_1(\mathcal F_{1,p^\infty}))\cong T_p,
\end{equation}
using \cite[62.J]{Fuchs}.
In particular, we see that $K_i(A\otimes \mathcal F_{1,p^\infty})$ is an extension of $p$-groups, so it is itself a $p$-group, and therefore by Lemma \ref{lem:NoCycSummand}, $K_i(A\otimes \mathcal F_{1,p^\infty})$ is either $0$ or $\Q_{p^\infty}/\Z$.
By \eqref{eq:Necessary1}, it follows that $T_p$ is a quotient of either $0$ or $\Q_{p^\infty}/\Z$, and therefore it is isomorphic to either $0$ or $\Q_{p^\infty}/\Z$.

For step (iii), we first note that since $T_{G_i}$ is divisible and $G_i/T_{G_i}$ is a subgroup of $\Q$, it follows using \cite[Lemma 1.1]{aifK} that
\[ G_i \cong T_{G_i} \oplus G_i/T_{G_i}. \]

If $G_1/T_{G_1} \neq 0$ then since this group is torsion-free, we have that $(G_1/T_{G_1}) \otimes (G_1/T_{G_1})$ is a nonzero direct summand of $K_0(A\otimes A)$ (by Lemma \ref{lem:TensorKComp}).
However by \eqref{eq:KunnethFlipDiag1}, $K_0(\sigma_{A,A}) = -\sigma_{(G_1/T_{G_1}),(G_1/T_{G_1})}$ on this direct summand, which is different from the identity map.
This contradicts that $A$ has approximately inner flip.
Therefore, $G_1/T_{G_1} = 0$, i.e., $G_1=T_{G_1}$.

Likewise, if $T_{G_0} \neq 0$ then $\Tor(T_{G_0},T_{G_0}) \cong \Q_{m_0}/\Z$ would be a nonzero direct summand of $K_1(A\otimes A)$, and by \eqref{eq:KunnethFlipDiag2}, $K_1(\sigma_{A,A}) = -\eta_{T_{G_0},T_{G_0}}$ on this direct summand.
By \cite[Proposition 5.1]{aifK}, this differs from the identity map, so this contradicts approximately inner flip.
We conclude that $T_{G_0}=0$, i.e., $G_0$ is torsion-free and hence a subgroup of $\Q$.

Relabelling, we can summarize by saying that $G_0$ is either $0$ or $G_0=\Q_n$ for some supernatural number $n$, and $G_1=\Q_m/\Z$ for some supernatural number $m$ of infinite type.
Finally, by Lemma \ref{lem:BasicRestrictions}, we have $G_0\otimes G_1=0$, which implies that $m$ must divide $n$.
\end{proof}

\begin{proof}[Proof of Theorem \ref{thm:MainThmCorrected}]
(i) $\Rightarrow$ (iii) is a combination of Theorem \ref{thm:NecessaryCorrected} and \cite[Propositions 2.7, 2.8, 2.10]{EffrosRosenberg}.
(iii) $\Rightarrow$ (iv) uses classification of $\cstar$-algebras (see \cite[Remark 2.1]{aifK}).
(iv) $\Rightarrow$ (ii) is Corollary \ref{cor:Suff}, and (ii) $\Rightarrow$ (i) is immediate.
\end{proof}

Finally, we explain why \cite[Lemmas 6.2 and 6.4]{aifK} are correct as stated (although their proofs are not).
We can easily see from the form of $K_*(A)$ given by Theorem \ref{thm:NecessaryCorrected} that \cite[Lemma 6.2]{aifK} holds.

For \cite[Lemma 6.4]{aifK}, part (i) is an obvious consequence of Theorem \ref{thm:NecessaryCorrected}.
For (ii), let us assume that $K_*(A) = G_p$ and that $A\otimes A$ has approximately inner flip.
Then by the K\"unneth formula, we find that $K_*(A\otimes A)$ is a nonzero $p$-group, from which it follows by Theorem \ref{thm:NecessaryCorrected} that $K_0(A\otimes A)=0$ and $K_1(A\otimes A)\cong \Q_{p^\infty}$.
Consequently, $K_0(A)$ and $K_1(A)$ must be $p$-divisible (or else $K_i(A)\otimes K_i(A)\neq 0$ which would imply $K_0(A\otimes A)\neq 0$ by the K\"unneth formula).
Using this and the K\"unneth formula again, we see that $\Q_{p^\infty} \cong K_1(A\otimes A)\cong \Tor(K_0(A),K_0(A))\oplus \Tor(K_1(A),K_1(A))$.
As $\Q_{p^\infty}$ is directly indecomposable, $\Tor(K_i(A),K_i(A))\cong 0$, which implies that $K_i(A)=0$ (by the Cartan--Eilenberg exact sequence for Tor).
By \cite[Corollary 27.4]{Fuchs}, it follows that $G_p=K_{1-i}(A)$ is $\Q_{p^\infty}/\Z$.

\end{document}